\documentclass[11pt,reqno]{amsart}
\usepackage{amsmath,amssymb,amsxtra,latexsym,amsthm,amscd,amsfonts,mathrsfs,pb-diagram}
\usepackage[colorlinks=true,linkcolor=magenta,citecolor=blue]{hyperref}
\usepackage[margin=1in]{geometry}
\usepackage{enumitem}
\usepackage{mathabx}
\usepackage{epsfig}
\newtheorem{example}{Example}
\newtheorem{corollary}{Corollary}

\usepackage{color}

\begin{document}
	\title[\hfil ]{EFFECT OF DIFFERENT ADDITIONAL $L^{m}$ REGULARITY ON SEMI-LINEAR DAMPED $\sigma$-EVOLUTION MODELS}
\thanks{$^\star$ Corresponding author}
	%\author[Khaldi Said]
	{}  % in alphabetical order
	%\email{said.khaldi@univ-sba.dz,~~saidookhaldi@gamil.com.}
	%\author[ ]
	{}  % in alphabetical order
	%\email{}
	%\subjclass[....]{.... }
	%\keywords{Global existence; energy estimates; nonlinear memory; power nonlinearity.}
\subjclass[2010]{35A01, 35L30, 35B33, 35B45, 35B44}

\keywords{$\sigma$-evolution equation, structural damping; global existence, critical exponent, additional regularity.}

	\maketitle
	\centerline{\scshape Said Khaldi$^\star$}
	\medskip
	
	{\footnotesize
	\centerline{Laboratory of Analysis and Control of PDEs, Djillali Liabes University}
	\centerline{P.O. Box 89, Sidi Bel Abbes 22000, Algeria}
	\centerline{Emails: saidookhaldi@gmail.com,\ said.khaldi@univ-sba.dz}
	}
\medskip

\centerline{\scshape Fatima Zahra Arioui}
\medskip
{\footnotesize
	\centerline{Laboratory of Statistics and Stochastic Processes, Djillali Liabes University} 
	\centerline{P.O. Box 89, Sidi-Bel-Abbes 22000, Algeria}
	\centerline{Email: ariouifatimazahra@gmail.com}
} 
% Do not forget to end the {\footnotesize by the sign }

\begin{abstract}
The motivation of the present study is to discuss the global (in time) existence of small data solutions to the following semi-linear structurally damped $\sigma$-evolution models:
\begin{equation*}
\partial_{tt}u+(-\Delta)^{\sigma}u+(-\Delta)^{\sigma/2}\partial_{t}u=\left|u\right| ^{p}, \ \sigma\geq 1, \ \ p>1,
\end{equation*} where the Cauchy data $(u(0,x), \partial_{t}u(0,x))$ will be chosen from energy space on the base of $L^{q}$ with different additional $L^{m}$ regularity, namely
\begin{equation*}
u(0,x)\in H^{\sigma,q}(\mathbb{R}^{n})\cap L^{m_{1}}(\mathbb{R}^{n}) , \ \ \partial_{t}u(0,x)\in L^{q}(\mathbb{R}^{n})\cap L^{m_{2}}(\mathbb{R}^{n}), \ \ q\in(1,\infty),\ \ m_{1}, m_{2}\in [1,q).
\end{equation*}
Our new results will show that the critical exponent which guarantees the global (in time) existence is really affected by these different additional regularities and will take \textit{two different values} under some restrictions on $m_{1}, m_{2}$, $q$, $\sigma$ and the space dimension $n\geq1$. Moreover, in each case, we have no loss of decay estimates of the unique solution with respect to the corresponding linear models.
\end{abstract}
\numberwithin{equation}{section}
\newtheorem{theorem}{Theorem}[section]
\newtheorem{lemma}[theorem]{Lemma}
\newtheorem{definition}[theorem]{Definition}
\newtheorem{remark}[theorem]{Remark}
\newtheorem{proposition}[theorem]{Proposition}
\allowdisplaybreaks		
\section{Introduction}
In this paper, we consider the following semi-linear structurally damped $\sigma$-evolution models: 
\begin{equation}\label{1.1}
\partial_{tt}u(t,x)+(-\Delta)^{\sigma}u(t,x)+(-\Delta)^{\sigma/2}\partial_{t}u(t,x)=\left|u(t,x)\right| ^{p},
\end{equation}
where the Cauchy data are from the space $H^{\sigma,q}(\mathbb{R}^{n}) \times L^{q}(\mathbb{R}^{n})$ with different additional $L^{m}$ regularity:
\begin{equation}\label{1.2}
u(0,x)\in H^{\sigma,q}(\mathbb{R}^{n})\cap L^{m_{1}}(\mathbb{R}^{n}) , \ \ \partial_{t}u(0,x)\in L^{q}(\mathbb{R}^{n})\cap L^{m_{2}}(\mathbb{R}^{n}),
\end{equation}
with the parameters:
$$t\in[0,\infty),\ x\in\mathbb{R}^{n},\ n\geq 1, \ \sigma \in [1,\infty),\  q\in(1,\infty), \  m_{1}, m_{2}\in [1,q), \ p\in(1,\infty).$$

At first, the presence of the Laplacian operators with fractional orders in the above models make them describe many practical applications in the real world, this is due to the fact that the integral representation of the fractional Laplacian operator that contains a singular kernel  can represent the memory or hereditary process, we recall the following example when $\sigma=1$ in (\ref{1.1}), the obtained model without non-linearity is frequently used in the determination of lifespan for primary or rechargeable batteries, for more details one can see \cite{fanglureissig} and references therein.

In recent years, after the pioneering work \cite{fujita} of Hiroshi Fujita in 1966, many mathematicians focus their research works on finding or improving the so-called critical exponent $p_{c}$ of some nonlinear Cauchy problems. In general, this critical exponent $p_{c}$ is exactly a threshold that divides the range of $p\in(1,\infty)$ into two parts such that when:
\begin{itemize}
	\item $1 < p <p_{c}$ there exist arbitrarily small Cauchy data, such that there exists no global (in time) weak solution, only local existence can be proved.
	\item $p_{c}<p<\infty$ then there exist global (in time) small data Sobolev or energy solutions.
\end{itemize} 

Sometimes, the critical case $p=p_{c}$ belongs to the sets of global existence or blow-up. Moreover, the well-known power nonlinearity $|u|^{p}$ is one of some types of nonlinearities such that when $p \in (p_{c}, \infty)$, the decay estimates for the unique solution are exactly coincide with those for solutions to the linear problem.

Here, let us recall the derived critical exponent for the problem (\ref{1.1}) with Cauchy data in the energy space $H^{\sigma}(\mathbb{R}^{n}) \times L^{2}(\mathbb{R}^{n})$ with the \textit{same} additional $L^{m}$ regularity, that is (\ref{1.2}) with $q=2$ and $m_{1}=m_{2}=m \in[1,2)$, the authors in \cite{phamkainanereissig, d'Abbiccoebert, daothesis, daoreissigblowup} found the following  expression:
\begin{equation}\label{1.3}
p_{c}(n,m,\sigma)=1+\frac{2m\sigma}{n-m\sigma}, \ m\in[1,2).
\end{equation}
 
We note that the idea of choosing the Cauchy data from the energy space with the same additional $L^{m}$ regularity has its origin in the work \cite{ikehataohta}, where the authors studied the Cauchy problem for semi-linear dissipative wave equations:
$$w_{tt}-\Delta w+w_{t}=|w|^{p-1}w,$$ and Cauchy data as in (\ref{1.2}) with $\sigma=1$, $q=2$, $m_{1}=m_{2}=m\in[1,2]$, they first showed the direct influence of the parameter $m$ on the Fujita exponent $p_{F}(n)=1+2/n$, more precisely, they found the following modified Fujita exponent $p_{F}(n/m)=1+2m/n$.

From the above results, one can easily see that when the Cauchy data have a unified additional regularity, then one can expect at least one critical exponent, for more results about the equation (\ref{1.1}) see \cite{daothesis, daoreissigl1estimates, daoreissiganapplication}. So, it is interesting to study the models (\ref{1.1}) with a generalized class of Cauchy data in (\ref{1.2}).

Our essential goal is to show a new influence of different additional $L^{m}$ regularity on the critical exponents (\ref{1.3}), to do this, we would like to prove the global (in time) existence of small data solutions to (\ref{1.1})-(\ref{1.2}). By standard tools such as the Banach fixed point theorem, Gagliardo-Nirenberg inequality and some mixed $(L^{m}\cap L^{q})-L^{q}$ linear estimates, we find the following form of \textit{two different values}:
$$p_{1}(m_{2}, n, \sigma)=1+\frac{2m_{2}\sigma}{n-m_{2}\sigma}, \ \ \ \ \  p_{2}(m_{1}, m_{2}, n, \sigma)=\frac{m_{1}}{m_{2}}+\frac{m_{1}\sigma}{n},$$
under some restrictions on $m_{1}$, $m_{2}$, $q$,  $\sigma$ and $n$ provided in (\ref{2.6})-(\ref{2.7}). 

The paper is structured as follows: In Section \ref{Main tools} we introduce the tools that will be used, especially, the linear estimates which are a powerful tool. Section \ref{global existence results} contains our main results of global existence and their proofs. We end this paper with a conclusion and an open question. 
\section{Preliminaries
}\label{Main tools}
First, we set some notations that will be used here and in the following.
\subsection{Notations} 
\begin{itemize}
	\item For any vector $x=(x_{1},\cdots x_{n})\in\mathbb{R}^{n}$, we denote
	$|x|^{2}=(x_{1}^{2}+\cdots x_{n}^{2}).$
	\item $L^{m}(\mathbb{R}^{n})$ denote the usual Lebesgue spaces:
	$$f\in L^{m}(\mathbb{R}^{n}) \ \text{if}\ \   \|f\|^{m}_{L^{m}}=\int_{\mathbb{R}^{n}}^{}|f(x)|^{m}dx<\infty, \ m\in [1,q].$$
	\item As usual, $H^{\sigma, q}(\mathbb{R}^{n})$ denote the Sobolev spaces on the $L^{q}$ basis (see \cite{ebertreissig} p 445):
	$$f\in H^{\sigma,q}(\mathbb{R}^{n}) \ \text{if}\ \   \|f\|_{H^{\sigma,q}(\mathbb{R}^{n})}=\|\mathcal{F}^{-1}\left( (1+|\cdot|^{2})^{\frac{\sigma}{2}}\mathcal{F}(f)\right) \|_{L^{q}(\mathbb{R}^{n})}<\infty, \ \sigma>0.$$
	\item Throughout the present paper, any constant $c>0$ in the inequality $f(t) \leq c g(t)$ does not play any role in our study, hence we omit it and we write $f(t) \lesssim g(t)$ instead of $f(t) \leq c g(t)$. 
\end{itemize}
\subsection{Linear estimates}
Here, we begin by recalling the derived $(L^{m}\cap L^{q})-L^{q}$ estimates for solutions to the linear structurally damped $\sigma$-evolution models, that is,
\begin{equation}\label{2.1}
\begin{array}{lll}
\partial_{tt}u+(-\Delta)^{\sigma}u+(-\Delta)^{\sigma/2}\partial_{t}u=0, \ u(0,x)=u_{0}(x),\   \partial_{t}u(0,x)=u_{1}(x).
\end{array} 
\end{equation}
We have the following lemma.
\begin{lemma}[Corollary 16 \cite{phamkainanereissig}, Proposition 4.4 \cite{d'Abbiccoebert}]\label{LE1}
	Let $\sigma\geq 1$ in (\ref{2.1}), let $m\in [1,q)$ with $q\in(1,\infty)$. The solutions to (\ref{2.1}) satisfy the $(L^{m}\cap L^{q})-L^{q}$ estimates:
	\begin{align}
	\|u(t,\cdot)\|_{L^{q}(\mathbb{R}^{n})}&\lesssim(1+t)^{-\frac{n}{\sigma}\left(\frac{1}{m}-\frac{1}{q}\right)} \left\|u_{0}\right\| _{L^{m}(\mathbb{R}^{n})\cap L^{q}(\mathbb{R}^{n})}+(1+t)^{-\left( \frac{n}{\sigma}\left(\frac{1}{m}-\frac{1}{q}\right)-1\right) } \left\|u_{1}\right\| _{L^{m}(\mathbb{R}^{n})\cap L^{q}(\mathbb{R}^{n})},\label{2.2}
	\end{align}
	\begin{align*}
	\left\| \left( (-\Delta)^{\sigma
		/2}u, \partial_{t}u\right) (t,\cdot)\right\| _{L^{q}(\mathbb{R}^{n})}&\ \lesssim(1+t)^{-\frac{n}{\sigma}\left(\frac{1}{m}-\frac{1}{q}\right)-1} \left\|u_{0}\right\| _{L^{m}(\mathbb{R}^{n})\cap H^{\sigma, q}(\mathbb{R}^{n})}	\hspace{4cm}
	\end{align*}
	\begin{align}
	\hspace{7cm}+(1+t)^{-\frac{n}{\sigma}\left(\frac{1}{m}-\frac{1}{q}\right)} \left\|u_{1}\right\| _{L^{m}(\mathbb{R}^{n})\cap L^{q}(\mathbb{R}^{n})}, \label{2.3}
	\end{align}
	and the $L^{q}-L^{q}$ estimates:
	\begin{equation}
	\|u(t,\cdot)\|_{L^{q}(\mathbb{R}^{n})} \lesssim\|u_{0}\|_{L^{q}(\mathbb{R}^{n})}+(1+t)\|u_{1}\|_{L^{q}(\mathbb{R}^{n})}, \label{2.4}
	\end{equation}
	\begin{equation}
	\left\| \left( (-\Delta)^{\sigma/2}u, \partial_{t}u\right)(t,\cdot)\right\| _{L^{q}(\mathbb{R}^{n})} \lesssim(1+t)^{-1}\|u_{0}\|_{H^{\sigma,q}(\mathbb{R}^{n})}+\|u_{1}\|_{L^{q}(\mathbb{R}^{n})}, \label{2.5}
	\end{equation}
\end{lemma}
\begin{remark}
	Fortunately, since the $(L^{m}\cap L^{q})-L^{q}$ estimates (\ref{2.2})-(\ref{2.3}) are sharp, we can see from them that if $(m=q)$ the $L^{q}-L^{q}$ estimates (\ref{2.4})-(\ref{2.5}) are also sharp. This fact helps us to estimate the norms $$\|u(t,\cdot)\|_{L^{q}(\mathbb{R}^{n})} \ \text{and} \ \  \|((-\Delta)^{\sigma
		/2}u, \partial_{t}u)(t,\cdot)\|_{L^{q}(\mathbb{R}^{n})}$$ 
	when we treat the semi-linear problem, see for instance the proof of (\ref{3.11}), no loss of decay appears.
\end{remark}
Now, having in mind the data space (\ref{1.2}), we distinguish the following two cases.
\smallskip
\begin{corollary}\label{Linearestimates1}
	Let $\sigma\geq 1$ in (\ref{2.1}), let $q\in(1,\infty)$ and $m_{1}, m_{2}\in [1,q)$ with
	\begin{equation}\label{2.6}
	\left\lbrace  
	\begin{matrix}
	\frac{m_{2}q\sigma}{q-m_{2}}<n<\infty \hfill& {if }\ \ m_{1}\leq m_{2}, 
	&\cr
	\\
	\frac{m_{2}q\sigma}{q-m_{2}}<n\leq \frac{m_{1}m_{2}\sigma}{m_{1}-m_{2}} \hfill & {if } \ \ m_{2}<m_{1}. 
	&\cr
	\end{matrix}\right.  
	\end{equation} 
	Then, the solutions to (\ref{2.1}) satisfy the following estimates:
	\begin{align*}
	\|u(t,\cdot)\|_{L^{q}(\mathbb{R}^{n})}\lesssim(1+t)^{1-\frac{n}{\sigma}\left(\frac{1}{m_{2}}-\frac{1}{q}\right)}\left(  \left\|u_{0}\right\| _{L^{m_{1}}(\mathbb{R}^{n})\cap L^{q}(\mathbb{R}^{n})}+ \left\|u_{1}\right\| _{L^{m_{2}}(\mathbb{R}^{n})\cap L^{q}(\mathbb{R}^{n})}\right),
	\end{align*}
	\begin{equation*}
	\begin{split}
	\left\| ((-\Delta)^{\sigma
		/2}u,\partial_{t}u)(t,\cdot)\right\| _{L^{q}(\mathbb{R}^{n})} \lesssim& (1+t)^{-\frac{n}{\sigma}\left(\frac{1}{m_{2}}-\frac{1}{q}\right)}\Big(  \left\|u_{0}\right\| _{L^{m_{1}}(\mathbb{R}^{n})\cap H^{\sigma,q}(\mathbb{R}^{n})} \left\|u_{1}\right\| _{L^{m_{2}}(\mathbb{R}^{n})\cap L^{q}(\mathbb{R}^{n})}\Big).
	\end{split}
	\end{equation*}
\end{corollary}
\begin{corollary}\label{Linearestimates2}
	Let $\sigma\geq 1$ in (\ref{2.1}), let $q\in(1,\infty)$ and $m_{1}, m_{2}\in [1,q)$ with
	\begin{equation}\label{2.7}
	m_{2}<m_{1} \ \ \text{and}\ \ \ \frac{m_{1}m_{2}\sigma}{m_{1}-m_{2}}\leq n.
	\end{equation} 
	Then, the solutions to (\ref{2.1}) satisfy the following estimates:
	\begin{align*}
	\|u(t,\cdot)\|_{L^{q}(\mathbb{R}^{n})}\lesssim(1+t)^{-\frac{n}{\sigma}\left(\frac{1}{m_{1}}-\frac{1}{q}\right)}\left(  \left\|u_{0}\right\| _{L^{m_{1}}(\mathbb{R}^{n})\cap L^{q}(\mathbb{R}^{n})}+ \left\|u_{1}\right\| _{L^{m_{2}}(\mathbb{R}^{n})\cap L^{q}(\mathbb{R}^{n})}\right),
	\end{align*}
	\begin{equation*}
	\begin{split}
	\left\| ((-\Delta)^{\sigma
		/2}u,\partial_{t}u)(t,\cdot)\right\| _{L^{q}(\mathbb{R}^{n})} \lesssim& (1+t)^{-\frac{n}{\sigma}\left(\frac{1}{m_{1}}-\frac{1}{q}\right)-1}\Big(  \left\|u_{0}\right\| _{L^{m_{1}}(\mathbb{R}^{n})\cap H^{\sigma,q}(\mathbb{R}^{n})}+ \left\|u_{1}\right\| _{L^{m_{2}}(\mathbb{R}^{n})\cap L^{q}(\mathbb{R}^{n})}\Big).
	\end{split}
	\end{equation*}
\end{corollary}
The proof of these Corollaries can be directly obtained from the above Lemma (\ref{LE1}). Another important additional tool is the fractional Gagliardo-Nirenberg inequality (see \cite{ebertreissig}).
\begin{lemma}\cite{ebertreissig}\label{FGN}
	Let $1<q<\infty$, $\sigma>0$ and $s\in[0,\sigma)$. Then, the following fractional Gagliardo-Nirenberg inequality holds for all $y\in H^{\sigma}(\mathbb{R}^{n})$
	$$\|(-\Delta)^{s/2}y\|_{L^{q_{1}}(\mathbb{R}^{n})}\leq \|(-\Delta)^{\sigma/2}y\|_{L^{q_{2}}(\mathbb{R}^{n})}^{\theta}\,\|y\|_{L^{q_{2}}(\mathbb{R}^{n})}^{1-\theta},$$
	where $$\theta=\frac{n}{\sigma}\left(\frac{1}{q_{2}}-\frac{1}{q_{1}}+\frac{s}{n} \right)\in\left[ \frac{s}{\sigma},1\right].$$
\end{lemma}
Finally, we need the following integral inequality to deal with the Duhamel's principle, see for instence the proofs of (\ref{3.12})-(\ref{3.13}).
\begin{lemma}\cite{ebertreissig}\label{Integral inequality}
	Let $a, b\in\mathbb{R}$. Then, it holds
	$$\int_{0}^{t}(1+t-s)^{-a}(1+s)^{-b}ds\leq\left\{ 
	\begin{matrix}
	C(1+t)^{-\min\{a,b\}} \hfill& {if }\ \ \max\{a,b\}>1,&\cr C(1+t)^{-\min\{a,b\}}\log(2+t) & {if } \ \
	\max\{a,b\}=1,&\cr
	C(1+t)^{1-a-b}\hfill& {if }\ \ \max\{a,b\}<1.
	\end{matrix}\right.$$
	
\end{lemma} 
In the next section, we can state and prove our main results. By following the two cases mentioned above we have two principal results.

\section{Global existence results}
\label{global existence results}
In our theorems we will show the interactions between all the parameters $m_{1}, m_{2}$, $q$, $n$, $\sigma$ and $p$.
\begin{theorem}\label{global existence results 1}
	Let us consider the Cauchy problem (\ref{1.1})-(\ref{1.2}) with $\sigma\geq1$,  $p>1$ and $q\in(1,\infty)$, $m_{1}, m_{2} \in [1,q)$. Let us assume the following conditions  
	\begin{equation}\label{3.1}
	\left\lbrace  
	\begin{matrix}
	\frac{m_{2}q\sigma}{q-m_{2}}<n<\infty \hfill& {if }\ \ m_{1}\leq m_{2}, 
	&\cr
	\\
	\frac{m_{2}q\sigma}{q-m_{2}}<n\leq \frac{m_{1}m_{2}\sigma}{m_{1}-m_{2}} \hfill & {if } \ \ m_{2}<m_{1}. 
	&\cr
	\end{matrix}\right.  
	\end{equation} 
	Moreover, the exponent $p$ should satisfy:
	\begin{equation}\label{3.2}
	\left\lbrace  
	\begin{matrix}
	\frac{q}{m_{2}}\leq p \leq \frac{n}{n-q\sigma} \hfill & {if }\ \ q\sigma<n\leq \frac{q^{2}\sigma}{q-m_{2}}, 
	&\cr
	\\
	\frac{q}{m_{2}}\leq p <\infty \hfill & {if } \ \
	1\leq n \leq q\sigma, 
	&\cr
	\end{matrix}\right.  
	\end{equation}
	and
	\begin{equation}\label{3.3}  
	p>1+\frac{2m_{2}\sigma}{n-m_{2}\sigma}.
	\end{equation}\\
	Then, there exists a constant $\varepsilon_0>0$ such that for any data 
	$$( u(0,x),\partial_{t}u(0,x))  \in \mathcal{D}:= \left(H^{\sigma,q}(\mathbb{R}^{n})\cap L^{m_{1}}(\mathbb{R}^{n}) \right)\times\left( L^{q}(\mathbb{R}^{n})\cap L^{m_{2}}(\mathbb{R}^{n})\right)$$
	with $$\left\|( u(0,x),\partial_{t}u(0,x)) \right\|_{\mathcal{D}}<\varepsilon_0,$$ 
	we have a uniquely determined globally (in time) solution
	$$u\in\mathcal{C}^{0}\left([0,\infty), H^{\sigma,q}(\mathbb{R}^{n})\right)\cap \mathcal{C}^{1}\left([0,\infty),L^{q}(\mathbb{R}^{n})\right)$$
	to (\ref{1.1}). Furthermore, the solution satisfies the estimates:
	\begin{equation*}
	\left\|u(t,\cdot)\right\|_{L^{q}(\mathbb{R}^{n})} \lesssim (1+t)^{1-\frac{n}{\sigma}\left(\frac{1}{m_{2}}-\frac{1}{q}\right)}\left\|( u(0,x),\partial_{t}u(0,x)) \right\|_{\mathcal{D}},
	\end{equation*}
	\begin{equation*}
	\left\|((-\Delta)^{\sigma/2}u,\partial_{t}u)(t,\cdot)\right\|_{L^{q}(\mathbb{R}^{n})}\lesssim (1+t)^{-\frac{n}{\sigma}\left(\frac{1}{m_{2}}-\frac{1}{q}\right)}\left\|( u(0,x),\partial_{t}u(0,x)) \right\|_{\mathcal{D}}.
	\end{equation*}
\end{theorem}
Our second result is read as follows. 
\begin{theorem}\label{global existence results 2}
	Let us consider the Cauchy problem (\ref{1.1})-(\ref{1.2}) with $\sigma\geq1$,  $p>1$ and $q\in(1,\infty)$, $m_{1}, m_{2} \in [1,q)$. Let us assume the following conditions
	\begin{equation}\label{3.4}
	m_{2}<m_{1} \ \ \text{and}\ \ \ \frac{m_{1}m_{2}\sigma}{m_{1}-m_{2}}\leq n.
	\end{equation}
	We assume the same conditions for $p$ as in (\ref{3.2}). Moreover, we suppose
	\begin{equation}\label{3.5}  
	p>\frac{m_{1}}{m_{2}}+\frac{m_{1}\sigma}{n}.
	\end{equation}
	Then we have the same conclusion as in Theorem (\ref{global existence results 1}). Furthermore, the unique solution $u$ satisfies the estimates:
	\begin{equation*}
	\left\|u(t,\cdot)\right\|_{L^{q}(\mathbb{R}^{n})} \lesssim (1+t)^{-\frac{n}{\sigma}\left(\frac{1}{m_{1}}-\frac{1}{q}\right)}\left\|( u(0,x),\partial_{t}u(0,x)) \right\|_{\mathcal{D}},
	\end{equation*}
	\begin{equation*}
	\left\|((-\Delta)^{\sigma/2}u,\partial_{t}u)(t,\cdot)\right\|_{L^{q}(\mathbb{R}^{n})}\lesssim (1+t)^{-\frac{n}{\sigma}\left(\frac{1}{m_{1}}-\frac{1}{q}\right)-1}\left\|( u(0,x),\partial_{t}u(0,x)) \right\|_{\mathcal{D}}.
	\end{equation*}
\end{theorem}
\begin{remark}
	We note that the conditions (\ref{3.3}) and (\ref{3.5}) ensure that the decay estimates of the unique solution coincide with those for solutions to the corresponding
	linear Cauchy problem (\ref{2.1}). While, applying the Gagliardo-Nirenberg inequality from Lemma (\ref{FGN}) leads to the bounds on $p$ and $n$ in (\ref{3.2}).
\end{remark}
\begin{remark} 
	The first theorem shows that the critical exponent in (\ref{1.3}) remains the same  only with an effect of $m_{2}$, this is due to the fact that we used the time decay estimates related to $u_{1}$, thanks to the condition (\ref{3.1}). 
\end{remark}
\begin{remark} 
	The lower bound of the spatial dimension $n$ in (\ref{3.1}) guarantees a decreasing time estimate of $\|u(t,\cdot)\|_{L^{q}}$. In addition, this condition makes the exponent in (\ref{3.3})  well defined.
\end{remark}
\begin{example}
	Let us illustrate our results with the following model
	\begin{equation*}
	\partial_{tt}u(t,x)-\Delta u(t,x)+(-\Delta)^{1/2}\partial_{t}u(t,x)=\left|u(t,x)\right| ^{p}, \ \ \sigma=1,
	\end{equation*}
	\begin{equation*}
	u(0,x)\in H^{1}(\mathbb{R}^{n})\cap L^{m_{1}}(\mathbb{R}^{n}) , \ \ \partial_{t}u(0,x)\in L^{2}(\mathbb{R}^{n})\cap L^{m_{2}}(\mathbb{R}^{n}), \ \ q=2,
	\end{equation*}
	we will show the admissible ranges of $m_{1}, m_{2}$ and $p$ as follow.
	
	Let us fix $n=2$, we have the following admissible ranges: 
	\begin{itemize}
		\item if $1\leq m_{1}\leq m_{2}<\frac{4}{3}$, then we can apply Theorem \ref{global existence results 1} and we get the global (in time) existence of small data solutions for any $p\in\left(1+\frac{2m_{2}}{2-m_{2}}, \infty\right),$
		\item for all $m_{1}\in[1,2)$, if $1\leq m_{2}<\min\left\lbrace \frac{4}{3},m_{1}\right\rbrace $, then we can apply Theorem \ref{global existence results 1} and we get the global (in time) existence of small data solutions for any $p\in\left(1+\frac{2m_{2}}{2-m_{2}}, \infty\right).$
	\end{itemize} 
	Let us now fix $n=3$. We distinguish the following admissible ranges:
	\begin{itemize}
		\item if $1\leq m_{1}\leq m_{2}<\frac{6}{5}$, then we can apply Theorem \ref{global existence results 1} and we get the global (in time) existence of small data solutions for any $p\in\left(1+\frac{2m_{2}}{3-m_{2}}, \infty\right),$
		\item if $m_{1}\in\left[ 1,\frac{3}{2}\right] $ and $1\leq m_{2}<\min\left\lbrace m_{1}, \frac{6}{5}\right\rbrace $, then we can apply Theorem \ref{global existence results 1} and we get the global (in time) existence of small data solutions for any $p\in\left(1+\frac{2m_{2}}{3-m_{2}}, 3\right],$ 
		\item if $m_{1}\in\left[ \frac{3}{2},2\right) $ and $\frac{3m_{1}}{3+m_{1}}\leq m_{2}<\min\left\lbrace m_{1}, \frac{6}{5}\right\rbrace $, then we can apply Theorem \ref{global existence results 1} and we get the global (in time) existence of small data solutions for any $p\in\left(1+\frac{2m_{2}}{3-m_{2}}, 3\right],$
		\item if $m_{1}\in\left[ \frac{3}{2},2\right) $ and $1\leq m_{2}\leq\frac{3m_{1}}{3+m_{1}}$, then we can apply Theorem \ref{global existence results 2} and we get the global (in time) existence of small data solutions for any $p\in\left(\frac{m_{1}}{m_{2}}+\frac{m_{1}}{3}, 3\right].$
	\end{itemize}
\end{example}
Theorem \ref{global existence results 1} will be proved in three steps using the contraction mapping principle. \\

Since we are dealing with semi-linear Cauchy problems, we use the Banach's fixed point theorem
inspired from the book \cite[page 303]{ebertreissig}, see also the paper \cite{dabbiccoreissig}.  The left hand side of (\ref{1.1}) is linear with respect to $u$. Therefore, we can apply the Duhamel's principle to (\ref{1.1}) and write the solutions $u$ as an integral form:
\begin{equation}
\begin{split}
u(t,x)=& L_{1}(t,x)\ast_{(x)} u(0,x)+L_{2}(t,x)\ast_{(x)} \partial_{t}u(0,x)+ \int_{0}^{t}L_{2}(t-s,x)\ast_{(x)}|u(s,x)|^{p}ds,\label{3.6}
\end{split}
\end{equation}
where the kernels $L_{1}$, $L_{2}$ are given in the papers \cite{phamkainanereissig}, \cite{d'Abbiccoebert}, \cite{daothesis}. Let $T>0$, $B(T)$ is a family of Banach space defined as follows 
$$
B(T):=\mathcal{C}^{0}\left([0,T], H^{\sigma,q}\right)\cap \mathcal{C}^{1}\left([0,T],L^{q}\right).
$$
Let us define the operator $\mathcal{O} : u\in B(T)\longrightarrow \mathcal{O}u\in B(T)$ by
\begin{equation*}
u\longmapsto \mathcal{O}u=u^{L}+u^{N},   
\end{equation*} 
where 
$$
u^{L}(t,x)=L_{1}(t,x)\ast_{(x)} u(0,x)+L_{2}(t,x)\ast_{(x)} \partial_{t}u(0,x),
$$
$$u^{N}(t,x)=\int_{0}^{t}L_{2}(t-s,x)\ast_{(x)}|u(s,x)|^{p}ds.$$
The approach needs to define a suitable norm $\|\cdot\|_{B(T)}$ for the above Banach space $B(T)$ so that if the operator $\mathcal{O}$ satisfies the following two inequalities:
\begin{align}
\|\mathcal{O}u\|_{B(T)} &\lesssim\left\|( u(0,x),\partial_{t}u(0,x)) \right\|_{\mathcal{D}}+\|u\|_{B(T)}^{p}, \label{3.7} \\ 
\|\mathcal{O}u_{1}-\mathcal{O}u_{2}\|_{B(T)} &\lesssim \|u_{1}-u_{2}\|_{B(T)}\Big( \|u_{1}\|_{B(T)}^{p-1}+\|u_{2}\|_{B(T)}^{p-1}\Big), \label{3.8} 
\end{align}
then one can deduce the existence of a unique global (in time) solution satisfies $u=\mathcal{O}u$.

We clarify that the condition $ \left\|( u(0,x),\partial_{t}u(0,x)) \right\|_{\mathcal{D}}<\varepsilon$ makes the operator $\mathcal{O}$ maps balls of $B(T)$ into balls of $B(T)$.\\
In order to prove Theorem \ref{global existence results 1} we choose the following suitable norm
\begin{align}
\|u\|_{B(T)}&=\sup_{0\leq t\leq T}\Big((1+t)^{\frac{n}{\sigma}\left(\frac{1}{m_{2}}-\frac{1}{q}\right)-1}\|u(t,\cdot)\|_{L^{q}}+(1+t)^{\frac{n}{\sigma}\left(\frac{1}{m_{2}}-\frac{1}{q}\right)}\|((-\Delta)^{\sigma/2}u,\partial_{t}u)(t,\cdot)\|_{L^{q}}\Big), \label{3.9}
\end{align}
while to prove Theorem \ref{global existence results 2} we choose the following suitable norm 
\begin{align}
\|u\|_{B(T)}&=\sup_{0\leq t\leq T}\Big((1+t)^{\frac{n}{\sigma}\left(\frac{1}{m_{1}}-\frac{1}{q}\right)}\|u(t,\cdot)\|_{L^{q}}+(1+t)^{\frac{n}{\sigma}\left(\frac{1}{m_{1}}-\frac{1}{q}\right)+1}\|((-\Delta)^{\sigma/2}u,\partial_{t}u)(t,\cdot)\|_{L^{q}}\Big). \label{3.10}
\end{align}
Let us start with the proof of the first theorem. We divide the proof into three steps.
\begin{proof}[Proof of Theorem \ref{global existence results 1}]
	In the first step, we verify that $u^{L} \in B(T)$.\\
	\textit{Step 1:} Using Corollary \ref{Linearestimates1}, we deduce:
	\begin{align*}
	\|u^{L}\|_{B(T)}&=\sup_{0\leq t\leq T}\Big((1+t)^{\frac{n}{\sigma}\left(\frac{1}{m_{2}}-\frac{1}{q}\right)-1}\|u^{L}(t,\cdot)\|_{L^{q}}+(1+t)^{\frac{n}{\sigma}\left(\frac{1}{m_{2}}-\frac{1}{q}\right)}\|((-\Delta)^{\sigma/2}u^{L},\partial_{t}u^{L})(t,\cdot)\|_{L^{q}}\Big)
	\nonumber \\
	&\hspace{0.5cm}
	\lesssim \left\|(u(0,x),\partial_{t}u(0,x))\right\|_{\mathcal{D}}.
	\end{align*} 
	\textit{Step 2:} To conclude inequality (\ref{3.7}), we need to prove:
	\begin{equation}\label{3.11}
	\|u^{N}\|_{B(T)}\lesssim \|u\|_{B(T)}^{p}.
	\end{equation}
	To do this, we proceed as follows:
	\begin{align*}
	\|u^{N}(t,\cdot)\|_{L^{q}}&\lesssim\int_{0}^{t/2}(1+t-s)^{1-\frac{n}{\sigma}\left(\frac{1}{m_{2}}-\frac{1}{q}\right)}\left\|u(s,\cdot)\right\|^{p} _{L^{qp}\cap L^{m_{2}p}} ds
	\end{align*}
	\begin{align}
	+\int_{t/2}^{t}(1+t-s) \left\|u(s,\cdot)\right\|^{p} _{L^{qp}} ds, \label{3.12}
	\end{align}
	\begin{align*}
	\|((-\Delta)^{\sigma/2}u^{N},\partial_{t}u^{N})(t,\cdot)\|_{L^{q}}&\lesssim\int_{0}^{t}(1+t-s)^{-\frac{n}{\sigma}\left(\frac{1}{m_{2}}-\frac{1}{q}\right)} \left\|u(s,\cdot)\right\|^{p} _{L^{qp}\cap L^{m_{2}p}}ds
	\end{align*}
	\begin{align}
	+\int_{t/2}^{t}\left\|u(s,\cdot)\right\|^{p} _{L^{qp}} ds. \label{3.13}
	\end{align}
	Before applying the fractional Gagliardo-Nirenberg inequality from Lemma \ref{FGN} to estimate the norms 
	$\left\|u(s,\cdot)\right\|^{p} _{L^{rp}},  r=m_{2},q,$
	we should know from (\ref{3.9}) that:
	\begin{equation*}
	(1+s)^{\frac{n}{\sigma}\left(\frac{1}{m_{2}}-\frac{1}{q}\right)-1+\frac{a}{\sigma}}\|(-\Delta)^{a/2}u(s,\cdot)\|_{L^{q}}\lesssim\|u\|_{B(T)}, \ \ \ a=0,\sigma.
	\end{equation*}
	So, we can estimate the above norms as follows:
	\begin{equation*}
	\left\|u(s,\cdot)\right\| _{L^{r p}}^{p}\lesssim(1+s)^{-\frac{p(n-m_{2}\sigma)}{m_{2}\sigma}+\frac{n}{r\sigma}}\|u\|_{B(T)}^{p}, \ \ r=m_{2}, q,
	\end{equation*}
	where the conditions in (\ref{3.2}) are satisfied. Now, (\ref{3.12})-(\ref{3.13}) become:
	\begin{align*}
	\|u^{N}(t,\cdot)\|_{L^{q}}&\lesssim\|u\|_{B(T)}^{p}(1+t)^{1-\frac{n}{\sigma}\left(\frac{1}{m_{2}}-\frac{1}{q}\right)}\int_{0}^{t/2}(1+s)^{-\frac{p(n-m_{2}\sigma)}{m_{2}\sigma}+\frac{n}{m_{2}\sigma}}ds
	\end{align*}
	\begin{align*}
	\hspace{2cm}+ \|u\|_{B(T)}^{p}(1+t)^{-\frac{p(n-m_{2}\sigma)}{m_{2}\sigma}+\frac{n}{q\sigma}} \int_{t/2}^{t}(1+t-s)ds, 
	\end{align*}
	\begin{align*}
	\|((-\Delta)^{\sigma/2}u^{N},\partial_{t}u^{N})(t,\cdot)\|_{L^{q}}&\lesssim \|u\|_{B(T)}^{p}(1+t)^{-\frac{n}{\sigma}\left(\frac{1}{m_{2}}-\frac{1}{q}\right)}\int_{0}^{t/2}(1+s)^{-\frac{p(n-m_{2}\sigma)}{m_{2}\sigma}+\frac{n}{m_{2}\sigma}}ds
	\end{align*}
	\begin{align*}
	\hspace{3cm}+ \|u\|_{B(T)}^{p}(1+t)^{-\frac{p(n-m_{2}\sigma)}{m_{2}\sigma}+\frac{n}{q\sigma}} \int_{t/2}^{t}ds,
	\end{align*}
	where we used the equivalences
	$$(1+t-s)\approx (1+t) \ if\ s\in[0,t/2], \ \ (1+s)\approx (1+t)\ if\ s\in[t/2,t]. $$
	Thanks to Lemma \ref{Integral inequality} together with (\ref{3.3}), all the integrals can be estimated immediately 
	$$
	(1+t)^{\frac{n}{\sigma}\left(\frac{1}{m_{2}}-\frac{1}{q}\right)-1}\|u^{N}(t,\cdot)\|_{L^{q}}\lesssim\|u\|_{B(T)}^{p},
	$$ 
	$$ 
	(1+t)^{\frac{n}{\sigma}\left(\frac{1}{m_{2}}-\frac{1}{q}\right)}\|((-\Delta)^{\sigma/2}u^{N},\partial_{t}u^{N})(t,\cdot)\|_{L^{2}}\lesssim\|u\|_{B(T)}^{p}.
	$$
	In this way we proved
	$$\|\mathcal{O}u\|_{B(T)}=\|u^{L}+u^{N}\|_{B(T)} \lesssim\left\|( u(0,x),\partial_{t}u(0,x)) \right\|_{\mathcal{D}}+\|u\|_{B(T)}^{p}.$$
	\textit{Step 3:} Let us briefly show the proof of (\ref{3.8}). We choose two elements $u_{1}$, $u_{2}$ belong to $B(T)$, and we write
	$$\mathcal{O}u_{1}-\mathcal{O}u_{2}:=u_{1}^{N}-u_{2}^{N}=\int_{0}^{t}L_{2}(t-s,x)\ast( |u_{1}(s,x)|^{p}-|u_{2}(s,x)|^{p})ds.$$
	Here, we do as in step 2, we have again:
	$$ 	\|(u_{1}^{N}-u_{2}^{N})(t,\cdot)\|_{L^{q}}\lesssim\int_{0}^{t/2}(1+t-s)^{1-\frac{n}{\sigma}\left(\frac{1}{m_{2}}-\frac{1}{q}\right)}\left\||u_{1}(s,\cdot)|^{p}-|u_{2}(s,\cdot)|^{p}\right\|_{L^{m_{2}}\cap L^{q}}ds
	$$
	$$	+\int_{t/2}^{t}(1+t-s)\left\||u_{1}(s,\cdot)|^{p}-|u_{2}(s,\cdot)|^{p}\right\| _{L^{q}}ds, $$
	and
	$$	\|((-\Delta)^{\sigma/2}(u_{1}^{N}-u_{2}^{N}),\partial_{t}(u_{1}^{N}-u_{2}^{N}))(t,\cdot)\|_{L^{q}}\lesssim \int_{0}^{t/2}(1+t-s)^{-\frac{n}{\sigma}\left(\frac{1}{m_{2}}-\frac{1}{q}\right)}\left\||u_{1}(s,\cdot)|^{p}-|u_{2}(s,\cdot)|^{p}\right\| _{L^{m_{2}}\cap L^{q}}ds$$
	$$+\int_{t/2}^{t}\left\||u_{1}(s,\cdot)|^{p}-|u_{2}(s,\cdot)|^{p}\right\|_{L^{q}}ds.
	$$
	To calculate the above integrals, we use the H\"{o}lder's inequality which gives
	\begin{equation*}
	\||u_{1}(s,\cdot)|^{p}-|u_{2}(s,\cdot)|^{p}\|_{L^{r}}\lesssim \|(u_{1}-u_{2})(s,\cdot)\|_{L^{r p}}\left(\|u_{1}(s,\cdot)\|_{L^{r p}}^{p-1}+\|u_{2}(s,\cdot)\|_{L^{r p}}^{p-1} \right),
	\end{equation*}
	as well as use the definition of the norms
	$$\|u_{1}-u_{2}\|_{B(T)}, \ \ \|u_{1}\|_{B(T)}, \ \ \|u_{2}\|_{B(T)}.$$ and the fractional Gagliardo-Nirenberg inequality we get (\ref{3.8}). The proof of Theorem \ref{global existence results 1} is completed.
\end{proof}	
\begin{proof}[Proof of Theorem \ref{global existence results 2}]
	Now, let us briefly mention the proof of Theorem \ref{global existence results 2}.\\
	\textit{Step 1:} Using Corollary \ref{Linearestimates2} and (\ref{3.10}) we deduce again:
	\begin{align*}
	\|u^{L}\|_{B(T)}&=\sup_{0\leq t\leq T}\Big((1+t)^{\frac{n}{\sigma}\left(\frac{1}{m_{1}}-\frac{1}{q}\right)}\|u^{L}(t,\cdot)\|_{L^{q}}
	+(1+t)^{\frac{n}{\sigma}\left(\frac{1}{m_{1}}-\frac{1}{q}\right)+1}\|((-\Delta)^{\sigma/2}u^{L},\partial_{t}u^{L})(t,\cdot)\|_{L^{q}}\Big)\nonumber \\
	&\hspace{0.5cm} \lesssim \left\|(u(0,x),\partial_{t}u(0,x))\right\|_{\mathcal{D}}.
	\end{align*}
	\textit{Step 2:} To prove (\ref{3.11}), we come to estimate $$\left\|u(s,\cdot)\right\|^{p} _{L^{rp}}, \ \ r=m_{2},q.$$
	So, recalling again from (\ref{3.10})
	\begin{equation*}
	(1+s)^{\frac{n}{\sigma}\left(\frac{1}{m_{1}}-\frac{1}{q}\right)+\frac{a}{\sigma}}\|(-\Delta)^{a/2}u(s,\cdot)\|_{L^{q}}\lesssim\|u\|_{B(T)}, \ \ \ a=0,\sigma.
	\end{equation*}
	Hence, we have the following estimates 
	\begin{equation*}
	\left\|u(s,\cdot)\right\| _{L^{r p}}^{p}\lesssim(1+s)^{-\frac{np}{m_{1}\sigma}+\frac{n}{r\sigma}}\|u\|_{B(T)}^{p}, \ \ r=m_{2}, q,
	\end{equation*}
	where the conditions in (\ref{3.2}) are satisfied.
	Now, (\ref{3.12})-(\ref{3.13}) become
	\begin{align*}
	\|u^{N}(t,\cdot)\|_{L^{2}}&\lesssim\|u\|_{B(T)}^{p}(1+t)^{1-\frac{n}{\sigma}\left(\frac{1}{m_{2}}-\frac{1}{q}\right)}\int_{0}^{t/2}(1+s)^{-\frac{np}{m_{1}\sigma}+\frac{n}{m_{2}\sigma}}ds
	\end{align*}
	\begin{align*}
	+\|u\|_{B(T)}^{p}(1+t)^{-\frac{np}{m_{1}\sigma}+\frac{n}{q\sigma}} \int_{t/2}^{t}(1+t-s)ds, 
	\end{align*}
	\begin{align*}
	\|((-\Delta)^{\sigma/2}u^{N},\partial_{t}u^{N})(t,\cdot)\|_{L^{2}}&\lesssim\|u\|_{B(T)}^{p}(1+t)^{-\frac{n}{\sigma}\left(\frac{1}{m_{2}}-\frac{1}{q}\right)}\int_{0}^{t/2}(1+s)^{-\frac{np}{m_{1}\sigma}+\frac{n}{m_{2}\sigma}}ds
	\end{align*}
	\begin{align*}
	+\|u\|_{B(T)}^{p}(1+t)^{-\frac{np}{m_{1}\sigma}+\frac{n}{q\sigma}} \int_{t/2}^{t}ds.
	\end{align*}
	The condition (\ref{3.5}) ensures that the integral over $[0,t/2]$ is bounded. Thanks to (\ref{3.4}) we have directly 
	$$
	(1+t)^{\frac{n}{\sigma}\left(\frac{1}{m_{1}}-\frac{1}{q}\right)}\|u^{N}(t,\cdot)\|_{L^{q}}\lesssim\|u\|_{B(T)}^{p},
	$$ 
	$$ 
	(1+t)^{\frac{n}{\sigma}\left(\frac{1}{m_{1}}-\frac{1}{q}\right)+1}\|((-\Delta)^{\sigma/2}u^{N},\partial_{t}u^{N})(t,\cdot)\|_{L^{2}}\lesssim\|u\|_{B(T)}^{p}.
	$$
	In this way we proved inequality (\ref{3.7}). One can prove (\ref{3.8}) as above. Hence, the proof of Theorem \ref{global existence results 2} is completed.
\end{proof}
\section*{Conclusion}
In this paper, we proved how the critical exponent (\ref{1.3})  is really affected by the different additional $L^{m}$ regularity of the Cauchy data. In particular, our results coincide with previous results when $m_{1}=m_{2}$. We recall that the two exponents (\ref{3.3}) and (\ref{3.5}) are imposed only in order to have the global (in time) existence, they coincide in case 
$$m_{2}<m_{1}, \ \ n=\frac{m_{1}m_{2}\sigma}{m_{1}-m_{2}}.$$ So, we can ask the following question:
$$\textit{What exactly is the critical exponent of the Cauchy problem (\ref{1.1})-(\ref{1.2}) ?}$$
% References should be collected at the end of the paper and numbered in  %
% alphabetical order of the authors' names. Titles of journals should be  %
% abbreviated as in Mathematical Reviews. The preferred style is shown in %
% the examples below.                                                     %

\end{document}